\documentclass{amsart}

\usepackage{amsmath,amssymb,amsthm}
\usepackage{graphicx}
\usepackage{xcolor}
\usepackage{colortbl}
\usepackage{arydshln}
\usepackage{microtype}
\usepackage{hyperref}
\usepackage[numbers,sort&compress]{natbib}

\newtheorem{theorem}{Theorem}[section]
\newtheorem{lemma}[theorem]{Lemma}
\newtheorem{corollary}[theorem]{Corollary}

\theoremstyle{definition}
\newtheorem{definition}[theorem]{Definition}
\newtheorem{example}[theorem]{Example}

\numberwithin{equation}{section}

\begin{document}

\title{Trident Tableaux for Tree-Child Networks with One Reticulation Node:
A Bijection with Two-Wall Tableaux}

\author{Hexuan Liu}

\thanks{Department of Pure and Applied Mathematics, Waseda University,
3-4-1 Okubo, Shinjuku-ku, Tokyo 169-8555, Japan.
E-mail: lhx@ruri.waseda.jp.}

\thanks{The author is supported by JSPS KAKENHI Grant Number JP25KJ2168.}

\begin{abstract}
Motivated by a word encoding of tree-child networks, we introduce trident
tableaux, which are Young tableaux with a unique three-cell column satisfying certain conditions on successors. We construct a bijection between trident tableaux with $n+1$ columns and two-wall tableaux with $n$ columns, that is, two-row fillings with two designated columns in which vertical order is not imposed. The bijection matches the three-part decompositions of the two classes and shows that each class has cardinality $n(n+1)C_n/2$, where $C_n$ is the $n$-th Catalan number. In particular, it gives a trident-tableau interpretation of a shifted form of OEIS A002457. As a consequence, we obtain an exact enumeration of tree-child networks with one reticulation node that contain a trident, and show that their proportion among all tree-child networks with one reticulation node tends to $1/8$ as the number of leaves tends to infinity.
\end{abstract}

\maketitle
\pagestyle{plain}

\smallskip
\noindent\textbf{Keywords.}
Young tableaux, bijections, Catalan numbers, Chung--Feller theorem,
phylogenetic networks.

\section{Introduction}

Phylogenetic networks are used to represent evolutionary histories including hybridization events and gene transfer events that cannot be described by phylogenetic trees. Tree-child networks form one of the most well-studied subclasses of phylogenetic networks, from both algorithmic and enumerative points of view \cite{CardonaRosselloValiente2009,CardonaZhang2020}.

Throughout, phylogenetic networks are binary. A \emph{phylogenetic network} is a simple
directed acyclic graph with a unique root of indegree $0$ and outdegree $1$,
and leaves of indegree $1$ and outdegree $0$ that are bijectively labelled
by a finite set. For a directed edge $(u,v)$, we call $u$ a parent of $v$ and $v$ a child of $u$.
 Every other vertex is either a \emph{tree node}, of
indegree $1$ and outdegree $2$, or a \emph{reticulation node}, of indegree
$2$ and outdegree $1$. A phylogenetic network $N$ is a \emph{tree-child network} if every
non-leaf vertex in $N$ has a child that is not a reticulation node. Phylogenetic networks are
considered up to directed graph isomorphism preserving the root and the leaf
labels. 

For ranked tree-child networks, the notion of a trident, namely, the configuration consisting of a reticulation node whose child is a leaf and whose two parents each have a leaf as their other child, was introduced and studied in \cite{BienvenuLambertSteel2022,FuchsLiuYu2024}. We use the same definition of a trident for the tree-child networks considered here.

The tableau class introduced below is motivated by the binary case of the word encoding of tree-child networks in \cite{ChangFuchsLiuWallnerYu2024}. For a tree-child network with $n+1$ leaves and one reticulation node, the word produced by this encoding corresponds to a Young tableau with $n$ columns obtained from two rows, each with $n$ cells, by adding one cell below one of the columns \cite[Remark~3.3]{ChangFuchsLiuWallnerYu2024}. The leaf labels are recorded by a separate permutation and do not affect the associated tableau; the correspondence is described in detail in Section~\ref{sec:network-tableau}.

Accordingly, for $n\geq2$, let $\mathcal Y_n$ be the set of Young tableaux consisting of a top row and a middle row, each with $n$ cells, together with one additional cell below one of the columns. The cells are filled with $1,2,\ldots,2n+1$, with entries increasing from left to right along each row and from bottom to top along each column.

\begin{definition}
For an element of $\mathcal Y_n$, let $k<j<i$ be the entries in the unique three-cell column, read from bottom to top. For any entry other than the largest entry, we call the next integer its \emph{successor}. We call this column a \emph{trident column} if $k+1$ and $j+1$ lie in the top row, and either $i+1$ lies in the top row or $i$ is the largest entry. We refer to these requirements
collectively as the \emph{trident condition}.

An element of $\mathcal Y_n$ whose unique three-cell column is a trident column is called a \emph{trident tableau}. Let $\mathcal A_n$ denote the set of trident tableaux in $\mathcal Y_n$, and let $a_n=|\mathcal A_n|$.
\end{definition}

See Figure~\ref{fig:trident-tableaux} for two examples of trident tableaux and their trident columns.

\begin{figure}[ht]
    \centering
    \includegraphics[width=0.75\textwidth]{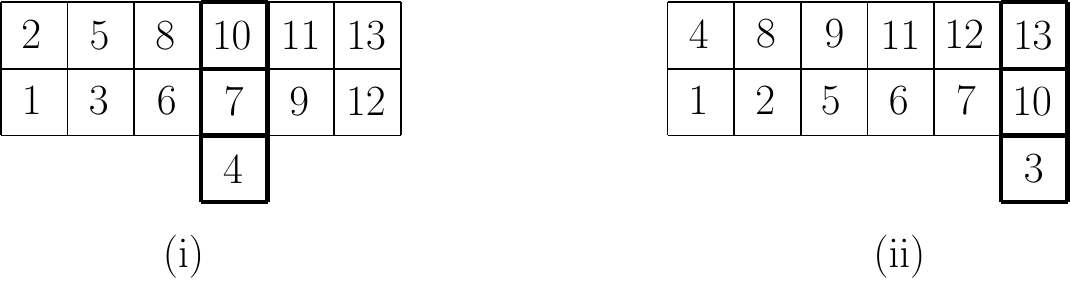}
    \caption{Two examples of trident tableaux. The bold column in each tableau is the trident column. In (i), $i+1$ lies in the top row, while in (ii), $i$ is the largest entry.}
    \label{fig:trident-tableaux}
\end{figure}

We next introduce the second tableau class used in our bijection. We follow the wall terminology of \cite{BanderierWallner2021}, with the usual vertical order suspended at a wall.

\begin{definition}
For $n\geq2$, a \emph{two-wall tableau} is a two-row filling with $n$ columns in which two distinct columns are designated as wall columns. The cells are filled with $1,2,\ldots,2n$. The entries increase from left to right along each row and from bottom to top in every column without a wall, while no vertical order is imposed in a wall column. Let $\mathcal B_n$ denote the set of two-wall tableaux with $n$ columns, and let $b_n=|\mathcal B_n|$.
\end{definition}

In Figure~\ref{fig:two-wall-tableaux}, tableaux (i) and (ii) correspond, under the bijection of Theorem~\ref{thm:main-bijection}, to tableaux (i) and (ii), respectively, in Figure~\ref{fig:trident-tableaux}.

\begin{figure}[ht]
    \centering
    \includegraphics[width=0.8\textwidth]{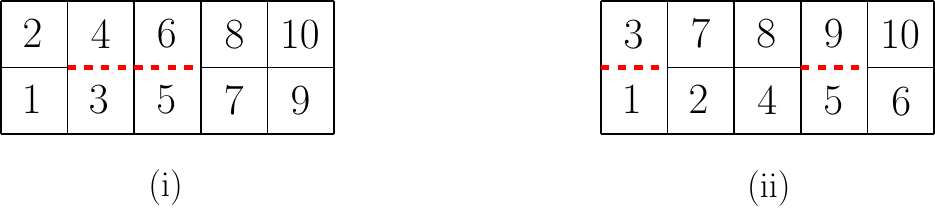}
    \caption{Two examples of two-wall tableaux corresponding to the
trident tableaux in Figure~\ref{fig:trident-tableaux}. The thick red dashed
horizontal lines indicate the walls.}
    \label{fig:two-wall-tableaux}
\end{figure}

Up to transposition, $\mathcal B_n$ is the case of two walls in the family studied in \cite{BanderierWallner2021}. In particular, \cite[Proposition~2.1]{BanderierWallner2021} gives $b_n=\binom{n+1}{2}C_n$, where $C_n$ is the $n$-th Catalan number. The class $\mathcal Y_n$ has been studied further in \cite{LiuWallnerYu2026}, including the limiting distributions of the column position of the bottom cell and the
entry it contains.

In this paper, Theorem~\ref{thm:main-bijection} gives an explicit bijection
between $\mathcal B_n$ and $\mathcal A_{n+1}$ that matches the three-part
decompositions of the two classes. The construction combines a Chung--Feller encoding of two-wall tableaux with an insertion that produces the trident column. Consequently,
$a_n=\binom{n}{2}C_{n-1}$ for $n\geq3$. Since $\mathcal A_2$ consists of a
single trident tableau, the formula holds for every $n\geq2$. The ordinary
generating function of $(a_n)_{n\geq2}$ is therefore
\[
A(z)=\sum_{n\geq2}a_nz^n
=\frac{z^2}{(1-4z)^{3/2}}
=z^2+6z^3+30z^4+140z^5+630z^6+2772z^7+\cdots.
\]
Thus, the relation $a_{n+1}=b_n$ identifies $(a_n)_{n\geq2}$ with a shifted
form of \href{https://oeis.org/A002457}{OEIS A002457}.

In \cite{BienvenuLambertSteel2022,FuchsLiuYu2024}, the expectation and asymptotic distribution of the number of tridents in a uniformly random ranked tree-child network are studied, where a ranking is a compatible ordering of the events of the network. Here we consider tree-child networks themselves, without adding such a ranking. Combining our bijection with the network--tableau correspondence of Section~\ref{sec:network-tableau} gives an exact enumeration of the tree-child networks with one reticulation node that contain a trident. For $m\geq3$, the number of such networks with $m$ labelled leaves is
\[
3\binom{m}{3}(2m-5)!!.
\]
Moreover, their proportion among all tree-child networks with $m$ labelled leaves and one reticulation node tends to $1/8$ as $m\to\infty$; see Corollary~\ref{cor:network-count}. 

Compared with the cited results on $\mathcal B_n$, $\mathcal Y_n$, and tridents in ranked tree-child networks, the new contributions are the trident subclass $\mathcal A_n$ and a decomposition-preserving bijection, which yields the
exact enumeration and limiting proportion above.

\section{The network--tableau correspondence}
\label{sec:network-tableau}

We now describe the correspondence between tree-child networks and tableaux used in this paper.

\subsection{The encoding}

A tree node is called \emph{free} if neither of its children is a reticulation node, and its two outgoing edges are called \emph{free edges}. We recall the binary case of the encoding in \cite[Theorem~3.4]{ChangFuchsLiuWallnerYu2024}, together with the reading convention used in its proof. Let $N$ be a tree-child network with $n+1$ leaves and one reticulation node. For each free tree node, choose one free edge, subdivide both free edges, and add a directed edge from the new vertex on the chosen edge to the other new vertex. The latter vertex thereby becomes a reticulation node. For any such collection of choices, let $N'$ denote the resulting network.

The construction from $N$ to $N'$ inserts vertices only on free edges. Since each parent of the reticulation node of $N$ has the reticulation node as a child, neither parent is free, and the reticulation node itself is not a tree node. Hence no vertex is inserted on any of the outgoing edges of these three vertices.

In $N'$, every tree node has exactly one reticulation node as a child. Hence $N'$ has a unique decomposition into \emph{path-components}, that is, maximal directed paths starting at the root or a reticulation node, ending at a leaf, and whose internal vertices are tree nodes. The incoming edges of the reticulation nodes are not contained in any path-component; they connect different path-components.

The path-component of $N'$ containing the root is assigned index $0$. Inductively, consider the unindexed path-components whose initial reticulation node has both parents in already indexed path-components. Order these path-components by the maximum of the indices of the two path-components containing the parents of their initial reticulation nodes, in increasing order. For those with the same maximum index $r$, read the path-component indexed by $r$ from its starting vertex towards its leaf. For each, take the last parent of its initial reticulation node encountered on this path-component, and order the path-components according to the order in which these parents are read. Assign consecutive indices to these path-components according to this order. Repeat until all path-components have been indexed.

Assign a distinct letter to each non-root path-component, and label its initial reticulation node and both parents with this letter. The letter assigned to the path-component with index $i$ corresponds to the $i$-th column of the associated tableau. For a reticulation node of $N'$ introduced by the construction at a free tree node of $N$, its two parents are the original tree node and the new tree node inserted on its chosen free edge. Hence these two vertices carry the same letter and are consecutive on a path-component; they are read as a single occurrence.

Reading these letters along the path-components in increasing order of their indices, each path-component from its starting vertex towards its leaf, gives a word. The leaves do not contribute letters to the word; their labels are read separately, in the same order, to give a permutation. The resulting word uses $n$ distinct letters: the letter associated with the reticulation node of $N$ occurs three times, while each of the remaining $n-1$ letters occurs twice \cite[Definition~3.1]{ChangFuchsLiuWallnerYu2024}.

If a letter occurring twice appears at positions $p_1<p_2$ in the word, then $p_1$ and $p_2$ are respectively the middle and top entries of its column in the associated tableau. If the letter occurring three times appears at positions $p_1<p_2<p_3$, then $p_1,p_2,p_3$ are respectively the bottom, middle, and top entries of the three-cell column \cite[Remark~3.3]{ChangFuchsLiuWallnerYu2024}.

Combining this encoding with the tableau correspondence above, \cite[Theorem~3.4]{ChangFuchsLiuWallnerYu2024} specializes to a bijection between tree-child networks with $n+1$ leaves and one reticulation node together with a choice of one free edge at each free tree node, and tableaux in $\mathcal Y_n$ together with an ordering of the $n+1$ leaf labels.

\subsection{The trident condition}

Every non-root path-component of $N'$ begins with a reticulation node. For the letter associated with such a path-component, all of its other occurrences arise from parents of that reticulation node, which lie in path-components of smaller index and are therefore read earlier. Thus the occurrence at the initial reticulation node is the final occurrence of its letter. Since the top entry of each column corresponds to the final occurrence of its letter, the top-row entries are precisely the positions in the word at which
non-root path-components begin.

\begin{lemma}\label{lem:trident-network}

Let $N$ be a tree-child network with one reticulation node, and let $T$ be
any Young tableau associated with $N$ under the encoding above. Then $N$
contains a trident if and only if the unique three-cell column of $T$ is a
trident column.

\end{lemma}

\begin{proof}

Let $k<j<i$ be the entries in the three-cell column. The entries $k$ and $j$
correspond to the occurrences of the letter associated with the reticulation
node at its two parents, while $i$ corresponds to the occurrence of that letter
at the reticulation node itself.

Consider either parent of the reticulation node. Since one of its children is the reticulation node, the tree-child property implies that its other child is either a leaf or a tree node. If the other child is a leaf, no further internal vertex is read in the current path-component after the occurrence at this parent. Since $k,j<i$, this occurrence is not the last entry of the word, so the next entry is the occurrence beginning the next path-component and hence lies in the top row by the observation above. If the other child is a tree node, the reading continues in the same path-component. The next entry is then not an occurrence beginning a path-component and therefore does not lie in the top row. Hence $k+1$ and $j+1$ lie in the top row precisely when the two parents each have a leaf as their other child.

Now consider the reticulation node itself. By the tree-child property, its unique child is either a leaf or a tree node. If it is a leaf, no further internal vertex is read in the current path-component after the reticulation node. If another path-component follows, then $i+1$ is the occurrence beginning that path-component and lies in the top row; if the current path-component is the last one, then $i$ is the largest entry. If the child is a tree node, the reading continues in the same path-component, so $i+1$ does not lie in the top row and $i$ is not the largest entry. Therefore, $N$ contains a trident if and only if the three-cell column is a trident column. 
\end{proof}

\medskip
\noindent\textbf{Choice-independence of the trident condition.}

Although the associated tableau may depend on the choices made at free tree nodes, whether its unique three-cell column is a trident column is independent of these choices. As noted above, neither parent of the reticulation node in $N$ is free, and the reticulation node itself is not a tree node, so no vertex is inserted on any outgoing edge of these three vertices when constructing $N'$. Thus, after the occurrence at either parent, the reading of $N'$ continues in the same path-component if and only if the other child of that parent is a tree node in $N$. Similarly, after the occurrence at the reticulation node, the reading continues in the same path-component if and only if its child is a tree node in $N$.

Different choices at free tree nodes may change the resulting network $N'$, and hence its path-component decomposition and the associated reading. In particular, when the child of the reticulation node is a leaf, they may determine whether the path-component containing the reticulation node is the last one read. If it is not, then $i+1$ lies in the top row; if it is, then $i$ is the largest entry. These are precisely the two alternatives allowed in the definition of a trident column. Thus, although the associated tableau may vary with the choices, the trident-column property depends only on the corresponding local configuration in $N$.

The local correspondence between a trident and its trident column is illustrated in Figure~\ref{fig:trident-correspondence}.

\begin{figure}[ht]     \centering     \includegraphics[width=0.95\textwidth]{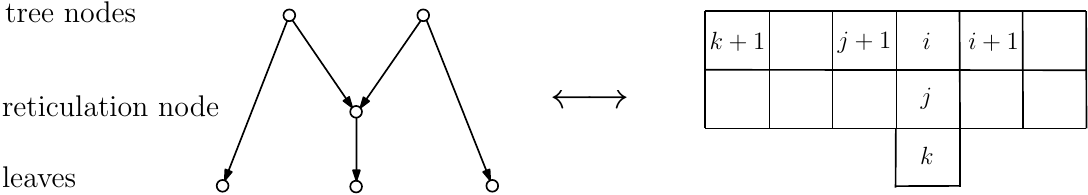}     

\caption{A trident in a tree-child network and its corresponding trident column. The letter associated with the reticulation node is read at its two parents and at the reticulation node itself, giving the three entries $k<j<i$ in the same column. Here the case in which $i+1$ lies in the top row is shown.}    

\label{fig:trident-correspondence} 

\end{figure}

\section{A decomposition of trident tableaux} \label{sec:trident-decomposition}

Throughout, column indices are counted from left to right. Let $A\in \mathcal A_n$, and let $c$ be the index of its trident column. Let $k<j<i$ be the entries in this column, read from bottom to top. Since $k+1$ and $j+1$ lie in the top row, let $q$ and $s$ be their respective column indices. Since $k+1<j+1$ and the top row is increasing, we have $q<s$.

The following lemma describes the possible positions of $q$, $s$, and $c$.

\begin{lemma} \label{lem:trident-decomposition} With the notation above, either $q<s<c$, or $q<s=c=n$. \end{lemma}

\begin{proof} Suppose first that $s>c$. The middle entry in the $s$-th column is larger than $j$, since the middle row is increasing, but smaller than its top entry $j+1$. This is impossible. Hence $s\leq c$.

Suppose that $s=c<n$. Then the top entry in the trident column is $i=j+1$. Since $c<n$, the entry $i$ is not the largest, so the trident condition implies that $i+1$ lies in the top row, and hence it is the top entry in the $(c+1)$-th column. The middle entry in the $(c+1)$-th column is larger than $j=i-1$ but smaller than $i+1$, so it must be $i$, which has already been used in the trident column. This is impossible. Therefore, $s=c$ can occur only when $c=n$. \end{proof}

We now divide $\mathcal A_n$ into three subclasses: 
\[ 
\begin{aligned} \mathcal A_n^{(0)}    &=\{A\in \mathcal A_n:q<s<c\},\\ \mathcal A_n^{(1)}    &=\{A\in \mathcal A_n:q<s=c=n,\ q<n-1\},\\ \mathcal A_n^{(2)}    &=\{A\in \mathcal A_n:q=n-1,\ s=c=n\}. 
\end{aligned} 
\] 

By Lemma~\ref{lem:trident-decomposition}, $\mathcal A_n$ is the disjoint union of $\mathcal A_n^{(0)}$, $\mathcal A_n^{(1)}$, and $\mathcal A_n^{(2)}$. For $e\in\{0,1,2\}$, let $a_n^{(e)}=|\mathcal A_n^{(e)}|$ denote the cardinality of $\mathcal A_n^{(e)}$.

In $\mathcal A_n^{(0)}$, the entries $k+1$ and $j+1$ both lie strictly to the left of the trident column. 

In $\mathcal A_n^{(1)}$ and $\mathcal A_n^{(2)}$, the trident column is the last column and $i=j+1=2n+1$. The two cases are distinguished according as $k+1$ lies strictly to the left of the top cell in the $(n-1)$-th column or is itself the top entry in the $(n-1)$-th column.

\section{A decomposition of two-wall tableaux} \label{sec:two-wall-tableaux}

Recall that $\mathcal B_n$ is the set of two-wall tableaux with $n$ columns. We call a column \emph{increasing} if its top entry is larger than its bottom entry, and \emph{decreasing} if its top entry is smaller than its bottom entry.

For $B\in \mathcal B_n$, read the entries $1,2,\ldots,2n$ in increasing order, writing $X$ for an entry in the bottom row and $Y$ for an entry in the top row. This gives a \emph{balanced word}, that is, a word with the same number of occurrences of $X$ and $Y$. Number the occurrences from left to right as $X_1,\ldots,X_n$ and $Y_1,\ldots,Y_n$. Since both rows are increasing, $X_r$ and $Y_r$ are respectively the bottom and top entries in the $r$-th column. Thus, the balanced word determines the two-row filling, and a two-wall tableau is determined by its balanced word together with its two wall columns. In particular, the $r$-th column is decreasing precisely when $Y_r$ occurs before $X_r$. Figure~\ref{fig:two-wall-example} illustrates this encoding.

\begin{figure}[ht]     \centering     \includegraphics[width=0.6\textwidth]{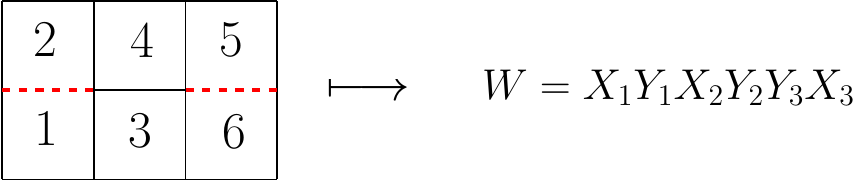}     
\caption{A two-wall tableau in $\mathcal B_3$ with walls in the first and
third columns, indicated by thick red dashed horizontal lines. The third
column is decreasing, since $Y_3$ occurs before $X_3$ in the corresponding
balanced word.}  \label{fig:two-wall-example} \end{figure}

Regarding $X$ as an up-step and $Y$ as a down-step, starting at height $0$, we obtain a lattice path; the step $Y_r$ ends below the horizontal axis precisely when $Y_r$ occurs before $X_r$. Thus, the number of decreasing columns is the Chung--Feller flaw statistic of the corresponding balanced word, namely, the number of indices $r$ for which $Y_r$ occurs before $X_r$; we call each such index a \emph{flaw}. Since every column without a wall must be increasing, any decreasing column is a wall column. Hence at most two columns are decreasing.

For $e\in\{0,1,2\}$, let $\mathcal B_n^{(e)}$ be the subset of $\mathcal B_n$ having exactly $e$ decreasing wall columns and let $b_n^{(e)}=|\mathcal B_n^{(e)}|$ denote the cardinality of $\mathcal B_n^{(e)}$. The set $\mathcal B_n$ is the disjoint union of $\mathcal B_n^{(0)}$, $\mathcal B_n^{(1)}$, and $\mathcal B_n^{(2)}$.

By the Chung--Feller theorem \cite{ChungFeller1949}, the number of balanced words with $n$ occurrences of each letter and exactly $e$ flaws is $C_n$. For a fixed such word, the decreasing columns are already determined, and each of them must be a wall column. If $e=0$, there is no such restriction, so the two walls can be placed in any two of the $n$ columns. If $e=1$, the unique decreasing column must be one wall, and the other wall can be placed in any of the remaining $n-1$ columns. If $e=2$, the two decreasing columns must be exactly the two wall columns. Hence 
\[
b_n^{(0)}=\binom{n}{2}C_n,\qquad
b_n^{(1)}=(n-1)C_n,\qquad
b_n^{(2)}=C_n.
\]

Summing the three cardinalities gives $b_n=\binom{n+1}{2}C_n$, in agreement
with \cite[Proposition~2.1]{BanderierWallner2021}. These three subclasses will
be matched with $\mathcal A_{n+1}^{(0)}$, $\mathcal A_{n+1}^{(1)}$, and $\mathcal A_{n+1}^{(2)}$ in Theorem~\ref{thm:main-bijection}.

\section{The bijection} \label{sec:bijection}

\subsection{The Chung--Feller encoding}

Let $\mathcal{D}_n$ be the set of \emph{Dyck words} with $n$ occurrences of $X$ and $n$ occurrences of $Y$, that is, words in which every prefix contains at least as many $X$'s as $Y$'s. We first encode a two-wall tableau in $\mathcal B_n$ by an element of \[ \mathcal{E}_n  =\{(P,q,s):P\in\mathcal{D}_n,\ 1\leq q<s\leq n+1\}. \]
We use the Chung--Feller map to turn the balanced word into a Dyck word, with the number of decreasing columns encoded by the pair $(q,s)$.

For any balanced word, define the \emph{height} of a prefix to be the number of $X$'s minus the number of $Y$'s in that prefix. For a balanced word $U$ with $n$ occurrences of each letter, let $d_-(U)$ denote its number of flaws as defined in Section~\ref{sec:two-wall-tableaux}.
We use the following standard block-exchange form of the Chung--Feller
bijection \cite{Chen2008}. If $d_-(U)<n$, write
\[
U=\alpha X\beta Y\gamma,
\qquad
f(U)=\beta Y\alpha X\gamma,
\]
where the displayed $X$ is the first $X$ that changes the height from $0$
to $1$, and the displayed $Y$ is the first subsequent $Y$ that returns
the height to $0$. Then $d_-(f(U))=d_-(U)+1$.
Conversely, if $d_-(U)>0$, write
$U=\beta Y\alpha X\gamma$, where the displayed $Y$ is the first $Y$ that
changes the height from $0$ to $-1$, and the displayed $X$ is the first
subsequent $X$ that returns the height to $0$, and define
$f^{-1}(U)=\alpha X\beta Y\gamma$.
These two operations are mutually
inverse. Thus, for $0\leq e<n$, the map $f$ is a bijection from the
balanced words $U$ with $d_-(U)=e$ to those with $d_-(U)=e+1$.
Here $f^2$ and $f^{-2}$ denote two successive applications of $f$ and
$f^{-1}$, respectively.

\medskip
\noindent\textbf{Forward map.} Let $B\in \mathcal B_n$, and let $W$ be its balanced word. We define $(P,q,s)\in\mathcal{E}_n$ according to the number of decreasing wall columns.

\medskip \noindent\emph{No decreasing wall columns.} If $B\in \mathcal B_n^{(0)}$, let $q<s$ be the indices of the two wall columns and set $P=W$.

\medskip \noindent\emph{One decreasing wall column.} If $B\in \mathcal B_n^{(1)}$, let $v$ be the index of the unique decreasing wall column and let $r$ be the index of the other wall column. Set $P=f^{-1}(W)$ and $s=n+1$, and define $q=r$ if $r<v$, and $q=r-1$ if $r>v$.
Thus, after the $v$-th column has been removed, the other wall column is the $q$-th column in the remaining list.

\medskip \noindent\emph{Two decreasing wall columns.} If $B\in \mathcal B_n^{(2)}$, set $P=f^{-2}(W)$ and $(q,s)=(n,n+1)$.

\medskip
Since the remaining list of columns in the second case has length $n-1$, we have $q\leq n-1$, so this case cannot produce the pair $(n,n+1)$. The three cases correspond respectively to $q<s\leq n$, $q<n<s=n+1$, and $(q,s)=(n,n+1)$, and therefore exhaust $\mathcal{E}_n$.

For example, the tableau in Figure~\ref{fig:two-wall-example} has one decreasing wall column. The unique decreasing column is the third column, while the other wall is in the first column. Since $W=XYXYYX$, we obtain $P=f^{-1}(W)=XXYXYY$ and $(q,s)=(1,4)$.

\medskip
\noindent\textbf{Inverse of the encoding.} Let $(P,q,s)\in\mathcal{E}_n$. 

If $s\leq n$, take $W=P$ and place the walls in the $q$-th and $s$-th columns.

If $s=n+1$ and $q<n$, take $W=f(P)$ and let $v$ be the index of its unique flaw. Place one wall in the $v$-th column, and place the other wall in the $r$-th column, where $r=q$ if $q<v$, and $r=q+1$ if $q\geq v$.

Finally, if $(q,s)=(n,n+1)$, take $W=f^2(P)$ and let $v_1<v_2$ be the indices of its two flaws. Place the walls in the $v_1$-th and $v_2$-th columns. This construction gives a bijection between $\mathcal B_n$ and $\mathcal{E}_n$.

\subsection{The insertion and its inverse}

For $T\in\mathcal Y_{n+1}$, its \emph{row word} is obtained by reading the entries in increasing order and writing $X$, $Y$, or $\ast$ according as the entry lies in the middle, top, or bottom row. Number the occurrences of $X$ and $Y$ in the row word from left to right
as $X_1,X_2,\ldots$ and $Y_1,Y_2,\ldots$, respectively. By construction, the symbol corresponding to entry $r$ occupies the $r$-th position in the row word. If the $c$-th column is the three-cell column, its middle and top entries are represented by $X_c$ and $Y_c$, respectively. For $T$ to be a trident tableau, $\ast$ and $X_c$ each must be
immediately followed by an occurrence of $Y$, while $Y_c$ must either be
immediately followed by another occurrence of $Y$ or be the final symbol.

We next construct a trident tableau in $\mathcal A_{n+1}$ from $(P,q,s)\in\mathcal E_n$. For $1\leq r\leq n$, let $G_r$ denote the gap immediately before the $r$-th occurrence of $Y$ in $P$, and let $G_{n+1}$ denote the gap at the end of $P$. Throughout the construction, all gaps refer to the original word $P$. Figure~\ref{fig:bijection-example} illustrates the gaps and the three
insertion positions in one example.

\begin{figure}[ht]
    \centering
    \includegraphics[width=0.95 \textwidth]{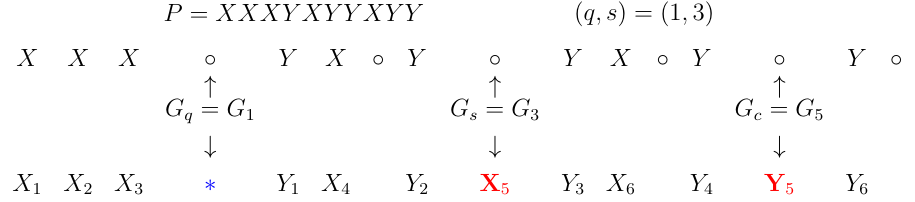}
    \caption{Insertion for
$P=XXXYXYYXYY$ and $(q,s)=(1,3)$.
The open circles indicate the gaps of $P$.
Here $c=5$, so the three insertion positions are $G_1$, $G_3$, and $G_5$;
the new occurrences of $X$ and $Y$ become $X_5$ and $Y_5$, respectively.
The inserted symbols are distinguished from the original word by both color
and typeface.}
    \label{fig:bijection-example}
\end{figure}

\noindent\textbf{The insertion.} Insert a symbol $\ast$ in $G_q$ and a new occurrence of $X$ in $G_s$. Let $c$ be the index of this new $X$ among all occurrences of $X$ in the resulting word. The resulting three-cell column will be the $c$-th column.
If $s\leq n$, then, since $P$ is a Dyck word, the $s$-th occurrence of $Y$ is preceded by at least $s$ occurrences of $X$. Hence the new $X$ has index $c>s$. If $s=n+1$, then $c=n+1$. We now insert a new occurrence of $Y$ in $G_c$, placing it after the new $X$ when $s=c=n+1$. In particular, if $s=n+1$, the newly inserted $X$ and $Y$ are the final two symbols of the resulting word. With the occurrences of $X$ and $Y$ numbered from left to right, the new $X$ and $Y$ are $X_c$ and $Y_c$, respectively, and hence occupy the same column.

Label the symbols of the resulting word by $1,2,\ldots,2n+3$ from left to right. Place the labels of the $X$'s in the middle row, the labels of the $Y$'s in the top row, and the label of $\ast$ in the bottom cell of the $c$-th column, and let $A$ denote the resulting tableau. Since the labels increase with the order of the symbols in the word, $X_r$ and $Y_r$ are respectively the middle and top entries in the
$r$-th column.

\medskip
\noindent\emph{Verification.}
Deleting $\ast$ from the resulting word leaves a Dyck word. Indeed, if $s\leq n$, inserting the new $X$ before the new $Y$ does not decrease the number of $X$'s relative to $Y$'s in any prefix. If $s=n+1$, the newly inserted $X$ and $Y$ are the final two symbols after $\ast$ is deleted. Thus every prefix still contains at least as many $X$'s as $Y$'s. It follows that the middle entry is smaller than the top entry in every column.

Let $k$, $j$, and $i$ be the labels of $\ast$, the new $X_c$, and the new $Y_c$, respectively. Since $q<s$ and either $s<c$ or $s=c=n+1$, the three inserted symbols occur in the order $\ast,X_c,Y_c$ in the resulting word. Hence $k<j<i$. Since $\ast$ is immediately followed by a $Y$, the entry $k+1$ lies in the top row. Similarly, the new $X_c$ is immediately followed by a $Y$, so $j+1$ lies in the top row. Finally, the new $Y_c$ is either immediately followed by another $Y$ or is the last symbol. Hence either $i+1$ lies in the top row or $i=2n+3$. Therefore, $A\in\mathcal A_{n+1}$.

\medskip
\noindent\emph{Preservation of the decomposition.}
In the row word, the $Y$ immediately following $\ast$ is $Y_q$. If $s<c$, the $Y$ immediately following $X_c$ is $Y_s$; if $s=c=n+1$, it is the newly inserted $Y_c$. Thus, $k+1$ and $j+1$ are the top entries in the $q$-th and $s$-th columns, respectively. Together with the properties of $c$ above, this shows that if $q<s\leq n$, then $A\in\mathcal A_{n+1}^{(0)}$; if $q<n<s=n+1$, then $A\in\mathcal A_{n+1}^{(1)}$; and if $(q,s)=(n,n+1)$, then $A\in\mathcal A_{n+1}^{(2)}$. See Appendix~\ref{app:examples} for one example of each subclass.

\medskip
\noindent\textbf{Inverse of the insertion.}
Let $A\in\mathcal A_{n+1}$, and let $c$ be the index of its three-cell column. In the row word of $A$, the middle and top entries of this column
are represented by $X_c$ and $Y_c$, respectively. Let $q$ be the index of the $Y$ immediately following $\ast$, and let $s$ be the index of the
$Y$ immediately following $X_c$. These indices are well defined by the trident condition and are precisely the column indices of the successors of the bottom and middle entries, respectively. By Lemma~\ref{lem:trident-decomposition}, either $q<s<c$ or
$q<s=c=n+1$.

Delete $\ast$, $X_c$, and $Y_c$, and call the remaining word $P$. Then
$P$ has $n$ occurrences of each letter and is a Dyck word. Indeed, the
row word of $A$ with $\ast$ deleted is a Dyck word since $A$ is a Young
tableau. For a prefix ending between $X_c$ and $Y_c$, at least $c$
occurrences of $X$ and at most $c-1$ occurrences of $Y$ have appeared,
so after $X_c$ is deleted, the prefix still contains at least as many
$X$'s as $Y$'s. Prefixes ending before
$X_c$ are unchanged, while for those ending after $Y_c$, one occurrence
of each letter is deleted.

We now verify that applying the insertion to $(P,q,s)$ reconstructs $A$. Since $q<s\leq c$, deleting $Y_c$ does not change the
index of the $Y$ immediately following $\ast$. If $s<c$, the same is true
for the $Y$ immediately following $X_c$; if $s=c=n+1$, deleting $X_c$
and $Y_c$ leaves the terminal gap $G_{n+1}$. Hence $\ast$ and $X_c$
occupy the gaps $G_q$ and $G_s$, respectively. If $s\leq n$, exactly
$c-1$ occurrences of $X$ precede $G_s$ in $P$, so reinserting $X$ in
$G_s$ gives $X_c$; if $s=n+1$, then $c=n+1$.

Finally, if $c\leq n$, then the top entry of the three-cell column is not
the largest entry, so the trident condition implies that $Y_c$ is
immediately followed by another $Y$, which becomes the $c$-th occurrence
of $Y$ after $Y_c$ is deleted. If $c=n+1$, then $Y_c$ is the final
symbol. Thus $Y_c$ occupies $G_c$, and the insertion reconstructs $A$. Combining this bijection with the Chung--Feller
encoding of the previous subsection gives the following result.

\begin{theorem} \label{thm:main-bijection} For every $n\geq2$, there is a bijection $\mathcal B_n\longrightarrow \mathcal A_{n+1}$ which restricts to bijections $\mathcal B_n^{(e)}\longrightarrow \mathcal A_{n+1}^{(e)}$ for $e=0,1,2$. Consequently, \[ a_{n+1}^{(0)}=\binom{n}{2}C_n,\qquad a_{n+1}^{(1)}=(n-1)C_n,\qquad a_{n+1}^{(2)}=C_n, \] and $a_{n+1}=b_n=\binom{n+1}{2}C_n$. \end{theorem}

For a fixed Dyck word $P$, the factors $\binom{n}{2}$, $n-1$, and $1$ count
the possible pairs $(q,s)$, that is, the column indices of $k+1$ and $j+1$:
in the first case any $q<s\leq n$ may be chosen; in the second,
$s=c=n+1$ and $q<n$ may be chosen; and in the third,
$(q,s)=(n,n+1)$ is forced.

We now translate this enumeration back to tree-child networks using the
correspondence of Section~\ref{sec:network-tableau}.

For $m\geq3$, let $\mathrm{TC}_{m,1}$ be the number of tree-child networks
with $m$ labelled leaves and one reticulation node, and let
$\mathrm{TR}_{m,1}$ be the number of those that contain a trident.

\begin{corollary}\label{cor:network-count}
For $m\geq3$,
\[
\mathrm{TR}_{m,1}=3\binom{m}{3}(2m-5)!!,
\qquad\text{and}\qquad
\frac{\mathrm{TR}_{m,1}}{\mathrm{TC}_{m,1}}\longrightarrow\frac18
\quad\text{as }m\to\infty.
\]
\end{corollary}

\begin{proof}
By \cite[Lemma~3.6]{ChangFuchsLiuWallnerYu2024}, a tree-child network
with $m$ leaves and one reticulation node has $m-2$ free tree nodes.
Indeed, if $t$ denotes the number of tree nodes, summing the outdegrees
and indegrees gives $1+2t+1=t+m+2$, so $t=m$; of these, exactly the two
parents of the reticulation node are non-free. Hence there are
$2^{m-2}$ choices of one free edge at each free tree node.
By the encoding of Section~\ref{sec:network-tableau},
\[
2^{m-2}\mathrm{TC}_{m,1}=m!\,|\mathcal Y_{m-1}|.
\]
By Lemma~\ref{lem:trident-network} and the choice-independence established
after its proof, the trident condition is preserved for all $2^{m-2}$
choices. Therefore, this correspondence restricts to
\[
2^{m-2}\mathrm{TR}_{m,1}=m!\,a_{m-1}.
\]
Hence,
\[
\mathrm{TR}_{m,1}
=\frac{m!}{2^{m-2}}a_{m-1}
=\frac{m!}{2^{m-2}}\binom{m-1}{2}C_{m-2}
=3\binom{m}{3}(2m-5)!!,
\]
where the second equality follows from
Theorem~\ref{thm:main-bijection}, together with $a_2=1$.

By \cite[Proposition~18]{CardonaZhang2020}, see also
\cite{FuchsGittenbergerMansouri2021},
$\mathrm{TC}_{m,1}=m(2m-1)!!-2^{m-1}m!$. For the second term,
\[
\frac{2^{m-1}m!}{m(2m-1)!!}
=\frac{4^m}{2m\binom{2m}{m}}
=O(m^{-1/2}),
\]
where we use $\binom{2m}{m}\sim4^m/\sqrt{\pi m}$. Hence
\[
\frac{\mathrm{TC}_{m,1}}{m(2m-1)!!}
=1+O(m^{-1/2})
\longrightarrow 1.
\]
From the formula for $\mathrm{TR}_{m,1}$ above,
\[
\frac{\mathrm{TR}_{m,1}}{m(2m-1)!!}
=\frac{(m-1)(m-2)}{2(2m-1)(2m-3)}
=\frac18+O(m^{-1})
\longrightarrow\frac18.
\]
Dividing the two estimates gives
\[
\frac{\mathrm{TR}_{m,1}}{\mathrm{TC}_{m,1}}
=\frac{1/8+O(m^{-1})}{1+O(m^{-1/2})}
=\frac18+O(m^{-1/2})
\longrightarrow\frac18.
\]
\end{proof}

\section{Conclusion}
The bijection in Theorem~\ref{thm:main-bijection} matches the three-part
decompositions of trident tableaux and two-wall tableaux, and leads to the
tree-child network enumeration and asymptotic proportion in
Corollary~\ref{cor:network-count}. A natural direction is to investigate whether the present bijection extends
to Young tableaux with several three-cell columns and several pairs of walls,
and to determine the network structures represented by such wall
configurations.

\section*{Acknowledgements}
The author thanks Dr. Michael Wallner for his valuable comments, and Dr. Momoko Hayamizu for her continued support.

\appendix

\section{Examples of the bijection}
\label{app:examples}

We write $n=3$ in all three examples, so the resulting trident tableaux have four columns. In the tableaux below, the walls are indicated by thick red dashed horizontal lines.

\begin{example}[No decreasing wall columns]
Consider the two-wall tableau

\[
\begin{array}{|c|c|c|}
\hline
3&4&6\\
\noalign{\global\arrayrulewidth=0.9pt}
\arrayrulecolor{red}\cdashline{1-2}[2pt/1.5pt]
\noalign{\global\arrayrulewidth=0.4pt}
\arrayrulecolor{black}\cline{3-3}
1&2&5\\
\hline
\end{array}
\]

Both wall columns are increasing. The associated balanced word is
$W=XXYYXY$, so $P=W$ and $(q,s)=(1,2)$.

The second occurrence of $Y$ in $P$ is preceded by two occurrences of
$X$. Hence the new occurrence of $X$ is $X_3$, and therefore $c=3$.
We insert $\ast$ in $G_1$, the new $X_3$ in $G_2$, and the new $Y_3$
in $G_3$. Thus,
\[
XXYYXY
\longmapsto
XXY\textcolor{red}{X}YXY
\longmapsto
XXY\textcolor{red}{X}YX\textcolor{red}{Y}Y
\longmapsto
XX\textcolor{blue}{\ast}Y\textcolor{red}{X}YX
\textcolor{red}{Y}Y.
\]

After labelling the symbols from left to right, we
obtain
\[
\begin{array}{|c|c|c|c|}
\hline
4&6&8&9\\
\hline
1&2&5&7\\
\hline
\multicolumn{2}{c|}{}&3&\multicolumn{1}{c}{}\\
\cline{3-3}
\end{array}
\]

The entries in the three-cell column are $3<5<8$, and their successors $4$, $6$, and $9$ all lie in the top row. Since $q<s<c$, the resulting trident tableau lies in $\mathcal A_{4}^{(0)}$.

Conversely, in the resulting trident tableau, the $Y$ immediately following $\ast$ is $Y_1$, while the $Y$ immediately following $X_3$ is $Y_2$. Hence $(q,s)=(1,2)$. Deleting $\ast$, $X_3$, and $Y_3$
recovers $P=XXYYXY$. Since $s\leq n$, we have $W=P$ and place the walls in the first and second columns, recovering the original two-wall tableau.
\end{example}

\begin{example}[Exactly one decreasing wall column]
Consider the two-wall tableau

\[
\begin{array}{|c|c|c|}
\hline
2&3&6\\
\arrayrulecolor{black}\cline{1-1}
\noalign{\global\arrayrulewidth=0.9pt}
\arrayrulecolor{red}\cdashline{2-3}[2pt/1.5pt]
\noalign{\global\arrayrulewidth=0.4pt}
\arrayrulecolor{black}
1&4&5\\
\hline
\end{array}
\]

The second column is decreasing, whereas the third column is increasing. The associated balanced word is $W=XYYXXY$.
Applying the inverse Chung--Feller map once gives
$P=f^{-1}(W)=XXYYXY$.

The unique decreasing wall column is the second column, so $v=2$, and
the other wall is in the third column, so $r=3$. After the $v$-th column
is removed, the $r$-th column is the second column in the remaining list. Hence $q=2$ and $(q,s)=(2,4)$.

Since $s=n+1$, we append a new pair $XY$ to $P$ and insert $\ast$ in
$G_2$. Thus,
\[
XXYYXY
\longmapsto
XXYYXY\textcolor{red}{X}\textcolor{red}{Y}
\longmapsto
XXY\textcolor{blue}{\ast}YXY
\textcolor{red}{X}\textcolor{red}{Y}.
\]

After labelling the symbols from left to right, we obtain
\[
\begin{array}{|c|c|c|c|}
\hline
3&5&7&9\\
\hline
1&2&6&8\\
\hline
\multicolumn{3}{c|}{}&4\\
\cline{4-4}
\end{array}
\]

The entries in the three-cell column are $4<8<9$. The successors of
the bottom and middle entries are $5$ and $9$, respectively, both of
which lie in the top row, while $9$ is the largest entry. Since $s=c=4$ and $q<3$, the resulting trident tableau lies in $\mathcal A_{4}^{(1)}$.

Conversely, since $s=4=n+1$ and $q<3$, we apply $f$ once and recover
$W=f(P)=XYYXXY$. Its unique decreasing column is the second column, and
since $q=2$, the second column in the remaining list $(1,3)$ is the third column, which is therefore the other wall column.
\end{example}

\begin{example}[Exactly two decreasing wall columns]
Consider the two-wall tableau

\[
\begin{array}{|c|c|c|}
\hline
1&4&5\\
\noalign{\global\arrayrulewidth=0.9pt}
\arrayrulecolor{red}\cdashline{1-1}[2pt/1.5pt]
\noalign{\global\arrayrulewidth=0.4pt}
\arrayrulecolor{black}\cline{2-2}
\noalign{\global\arrayrulewidth=0.9pt}
\arrayrulecolor{red}\cdashline{3-3}[2pt/1.5pt]
\noalign{\global\arrayrulewidth=0.4pt}
\arrayrulecolor{black}
2&3&6\\
\hline
\end{array}
\]

Both wall columns are decreasing. The associated balanced word is
$W=YXXYYX$. Applying the inverse Chung--Feller map twice gives
$P=f^{-2}(W)=XXYXYY$. Since the two wall columns are already
determined by the two decreasing columns, we set $(q,s)=(3,4)$.

Since $s=n+1$, we append a new pair $XY$ to $P$ and insert $\ast$ in
$G_3$. Thus,
\[
XXYXYY
\longmapsto
XXYXYY\textcolor{red}{X}\textcolor{red}{Y}
\longmapsto
XXYXY\textcolor{blue}{\ast}Y
\textcolor{red}{X}\textcolor{red}{Y}.
\]

After labelling the symbols from left to right, we obtain
\[
\begin{array}{|c|c|c|c|}
\hline
3&5&7&9\\
\hline
1&2&4&8\\
\hline
\multicolumn{3}{c|}{}&6\\
\cline{4-4}
\end{array}
\]

The entries in the three-cell column are $6<8<9$. The successors of
the bottom and middle entries are $7$ and $9$, respectively, both of
which lie in the top row, while $9$ is the largest entry. Since $q=3$ and $s=c=4$, the resulting trident tableau lies in $\mathcal A_{4}^{(2)}$.

Conversely, the pair $(q,s)=(3,4)$ indicates that the original tableau
has two decreasing wall columns. Applying $f$ twice gives
$W=f^2(P)=YXXYYX$, whose decreasing columns are the first and third columns.
These are therefore the two wall columns.
\end{example}

\bibliographystyle{amsplain}
\bibliography{references}
\end{document}